\documentclass[12pt]{article}

\usepackage{verbatim}
\usepackage{enumerate}

\usepackage{color, float}

\usepackage{amssymb,amsthm,amsmath}
\usepackage{latexsym}
\usepackage{graphicx}
\usepackage{hyperref}
\usepackage{float}
\usepackage{algorithm}
\usepackage{algorithmic}
\usepackage{url}
\usepackage{float}
\usepackage{enumitem} 
\usepackage{subcaption}
\usepackage{url}
\usepackage{fullpage}

\newtheorem{theorem}{Theorem}
\newtheorem{lemma}[theorem]{Lemma}

\newtheorem{conj}[theorem]{Conjecture}
\newtheorem{cor}[theorem]{Corollary}
\newtheorem{obs}[theorem]{Observation}

\newtheorem{prop}[theorem]{Proposition}

\title{On the Regularity, Planarity and Edge Bounds of Link-irregular Graphs} 

\author{
	Alexander Bastien
	\qquad
	Omid Khormali\thanks{Corresponding author. Email: \texttt{ok16@evansville.edu}} \\
	\small Department of Mathematics\\[-0.8ex]
	\small University of Evansville\\[-0.8ex]
	\small Evansville, Indiana 47722, USA.\\
	\small \texttt{ab995@evansville.edu}\\
	\small \texttt{ok16@evansville.edu}
}

\begin{document}

	\maketitle
	
	\begin{abstract}
		A graph $G$ is a link-irregular graph if every two distinct vertices of $G$ have non-isomorphic links. The link of a vertex $v$ in $G$ is the subgraph induced by the neighbors of $v$ in $G$. Ali, Chartrand and Zhang [{\it Discussiones Mathematicae. Graph Theory}, 45(1) (2025) p.95] conjectured that there exists no regular link-irregular graph.
		In this paper, we show that the existence of an $r$-regular link irregular graph is very likely for large enough $r$. In particular, we provide a 7-regular link irregular graph on 12 vertices, which serves as a counterexample to the conjecture. 
		Additionally, we prove that no bipartite link-irregular graphs exist, and there are no regular link-irregular graphs on $n$-vertices for $n \leq 9$. Also, we determine upper and lower bounds for the number of edges of link-irregular graphs. Furthermore, we show the minimum number of edges in a link-irregular graph on the $n$ vertices is $\Omega(n\sqrt{\log n})$. Finally, we prove that all but finitely many link-irregular graphs are non-planar, and there is no regular link-irregular planar graphs.
	\end{abstract}
	
	
	\section{Introduction}
	A graph $G$ is defined as a pair $(V(G),E(G))$ such that $V(G)$ represents a vertex set and $E(G)$ represents an edge set. We denote the number of vertices in a graph $G$ by $n(G)$, and the number of edges in the graph $G$ by $e(G)$. Throughout this paper, we consider only simple graphs that have no loops or multiple edges. The degree of a vertex $u$, denoted $d(u)$, is the number of edges connected to $u$, and $D(G)$ denotes the set of degrees of the vertices of graph $G$. The degree sum formula states that $\sum_{v\in V(G)}d(v)=2e(G)$. We write $u \leftrightarrow v$ to indicate that the vertices $u$ and $v$ are adjacent. A graph is irregular if no two vertices in the graph have the same degree. However, it is known that no such graph exists because there is no simple graph in which all vertex degrees are distinct. However, link-irregular graphs go a step further by requiring that for every pair of distinct vertices $u,v\in V(G)$, the neighborhood subgraphs $G[N(u)]$ and $G[N(u)]$ are non-isomorphic. Link-irregularity is a slightly weaker condition than irregularity, and this difference allows an infinite number of link-irregular graphs to exist. That is:
	\[
	G[N(u)]\ncong G[N(v)]   \ \ \ \ \ \ \ \ \ \ \     \forall u,v\in V(G), u\neq v.
	\]
	
	Non-isomorphic graphs are graphs with different structures, and their vertex sets cannot be matched while keeping the same edges. Here, $N(u)$ is the set of vertices which are connected to the vertex u in a graph $G$, and $G[N(u)]$  is the induced subgraph on the vertices of $N(u)$ in the graph $G$. The link $L(v)$ of a vertex $v$ in a graph $G$ is the same as $G[N(v)]$, that is, $L(v) = G[N(v)]$. We name $v$ the link owner in the link $L(v)$. Note that in link-irregular graphs, each vertex has a structurally unique neighborhood, making them distinct from other types of irregular graphs.  \\
	
	Recently, Ali, Chartrand, and Zhang introduced the link-irregular graphs \cite{akbar_book}. Beyond initial definitions, little is known about the properties of these graphs. They obtained the following result from the existence of the link-irregular graph \cite{akbar_paper}.
	
	\begin{theorem}[Ali, Chartrand, and Zhang]\label{akbar1}
		There exists a link-irregular graph of order $n$ if and only if $n\geq 6$.
	\end{theorem}
	The authors provided the unique link-irregular graph with 6 vertices in \cite{akbar_paper}, and it is
	
	\begin{figure}[htbp]
		\centering
		\includegraphics[width=0.28\textwidth]{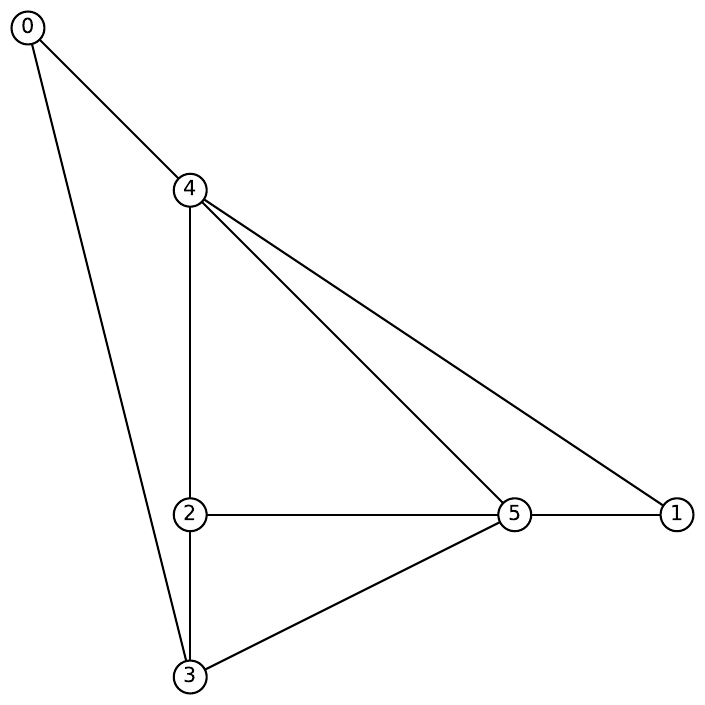}
		\caption{The unique link-irregular graph with 6 vertices \cite{akbar_paper}.}
		\label{fig:unique-6}
	\end{figure}
	
	The authors also investigated properties related to the degree sets of link-irregular graphs. The following results were established in~\cite{akbar_paper}. 
	\begin{theorem}[Akbar, Chartrand and Zhang]\label{degree-set}
		For each integer $n\geq 2$, there is no link-irregular graph $G$ of order $n$ such that $|D(G)| = n$ or $|D(G)| = n-1$.
	\end{theorem}
	
	\begin{theorem}[Akbar, Chartrand and Zhang]
		There exists a link-irregular graph $G$ of order $n$ such that
		$|D(G)| = n-2$ if and only if $n \geq 7$.
	\end{theorem}
	
	In addition, the regularity of link-irregular graphs was studied in~\cite{akbar_book}, leading to the following nonexistence result.
	\begin{theorem}[Akbar, Chartrand and Zhang]\label{no-regular}
		For $d = 0,1,2,3,4$, no $d$-regular link-irregular graph exists.
	\end{theorem}
	
	Then, they proposed the following conjecture \cite{akbar_book}.
	\begin{conj}\label{conjecture}
		There is no regular link-irregular graph.
	\end{conj}

	In the following of this paper, we disprove the above conjecture by presenting a counterexample. We also obtain several additional results concerning the girth, bounds on the number of edges, and planarity of link-irregular graphs. Furthermore, we explore the link-irregularity of trees and bipartite graphs.
	
	
	\section{Initial Results}
	
	We start by showing that trees and bipartite graphs are not link-irregular graphs. The following is the known result.
	
	\begin{theorem}[\cite{west}]\label{tree}
		Every tree with at least two vertices has at least two leaves.
	\end{theorem}
	
	Using this result, we have the following result for trees.
	\begin{theorem}
		Trees are not link-irregular graphs.
	\end{theorem}
	\begin{proof}
		By Theorem~\ref{tree}, each tree $T$ has at least two leaves $x, y \in V(T)$. Then, it is clear that $L(x) \cong L(y)$. Then $T$ is not a link-irregular graph.
	\end{proof}
	
	
	The following result is for bipartite graphs.
	
	\begin{theorem}\label{no-bipartite}
		There are no bipartite link-irregular graphs.
	\end{theorem}
	\begin{proof}
		The proof is by contradiction. Suppose there is a bipartite link-irregular graph. Then $\forall x, y \in V(G)$, we have $L(x) \ncong L(y)$. So since $L(x)$ and $L(y)$ are independent sets of vertices, then to have $L(x) \ncong L(y)$, the number of vertices in $L(x) \ncong L(y)$ should be different. Since $x$ and $y$ are selected arbitrarily, all of the vertices of $G$ should have different degrees. Then $|D(G)| = n$. But, this is a contradiction by Theorem ~\ref{degree-set}. Then $G$ is not a link-irregular graph.
	\end{proof}
	
	Note that the proof of Theorem~\ref{no-bipartite} gives an alternative proof for Theorem~\ref{tree} since the neighborhood of each vertex is an independent set. So, to have a link-irregular tree, all vertices should have a different degree, and this is impossible.\\
	
	Also, by Theorem~\ref{no-bipartite}, it is clear that there is no regular bipartite link-irregular graph. In addition, we have the following observation.
	\begin{obs}
		We have:
		\begin{enumerate}[label = (\alph*)]
			\item Hypercubes $Q_k = \{0, 1\}^k$ are not link-irregular graphs.
			\item If $G$ has no odd cycle, then $G$ is not a link-irregular graph.
		\end{enumerate}
	\end{obs}
	\begin{proof}
		Since hypercubes $Q_k$ are bipartite graphs, they are not link irregular graphs. Also, since $G$ has no odd cycle, $G$ is a bipartite graph. The result holds by Theorem~\ref{no-bipartite}.
	\end{proof}
	
	The following result states that any link-irregular graph must have an odd cycle of length 3.
	\begin{theorem}\label{girth}
		If $G$ is a link-irregular graph of order $n$, then its girth is $3$. 
	\end{theorem}
	\begin{proof}
		Let $G$ be a link-irregular graph with girth at least 4. Then the neighborhood of each vertex must be an independent set. Since $G$ is link-irregular graph, all vertices should have distinct degree. This is a contradiction by Theorem~\ref{degree-set}. Then the girth of $G$ is 3.
	\end{proof}
	
	\section{Disproving the Conjecture \ref{conjecture}}
	
	Note that if there is an $r$-regular link-irregular graph $G$ of order $n$, the number of vertices in $G$ cannot be more than the number of non-isomorphic graphs on $r$ vertices since the link of each vertex must be distinct. Let $g(r)$ be the number of non-isomorphic graphs on $r$ vertices. So, $n\leq g(r)$. We have 
	\[
	g(r) \geq \dfrac{2^{r \choose 2}}{r!}
	\]
	In this bound, the numerator counts the number of graphs on $r$ labeled vertices, and the denominator counts the number of ways to rearrange the labels on those vertices. \\
	
	In the following, we investigate the existence of regular link-irregular graphs using probabilistic methods. Although the Lov\'asz Local Lemma is a common tool in probabilistic existence proofs, it does not apply effectively to this problem due to the dependencies involved. Instead, we use the first and second moment methods to establish the existence of regular link-irregular graphs.

	\begin{theorem}\label{first-moment}
		For sufficiently large $g(r)$ and $n \ll g(r)$, an $r$-regular link irregular graph is highly likely to exist.
	\end{theorem}
	
	\begin{proof}
		For each $r$, we consider $r$-regular graphs $G$ on $n$ vertices where $n \ll g(r)$. For a vertex $v \in V(G)$, let $L(v)$ be the subgraph induced by its $r$ neighbors (the link of $v$). Since each $L(v)$ is a graph on $r$ vertices, there are at most $g(r)$ possible distinct link types.
		
		Suppose we assign to each vertex of $G$ one of the $g(r)$ possible link graphs uniformly at random. Let $X$ denote the number of vertices with unique links. Define indicator random variables $I_v$ for each vertex $v$ by:
		
		\[
		I_v = \begin{cases}
			1, & \text{if } L(v) \text{ is unique among all } L(u), u \neq v, \\
			0, & \text{otherwise}.
		\end{cases}
		\]
		
		Then $ X = \sum_{v \in V(G)} I_v $. While the events $I_v$ are not independent, we can estimate the expected value heuristically under the assumption that $n \ll g(r)$, which means the number of ways to assign links significantly exceeds the number of vertices.
		
		In such a setting, the probability that a randomly chosen link $L(v)$ coincides with another vertex’s link becomes small, and we expect that many vertices will have unique links. That is, $ \mathbb{E}[X] \approx n \cdot (1 - \frac{1}{g(r)})^{n-1} $. Using the approximation $\ln(1 - x) \approx -x$ for small $x$, we have
		\[
		\mathbb{E}[X] \approx n \cdot e^{-\frac{n-1}{g(r)}}.
		\]
		If $n \ll g(r)$, then $\frac{n-1}{g(r)}$ is small, and so $e^{-\frac{n-1}{g(r)}} \approx 1$, implying that $\mathbb{E}[X] \approx n$.
		
		So since $E[X]>1$, indicating that a $r$-regular graph with unique links is likely to exist.
	\end{proof}
	
	Next, we check the second-moment method, using the following result from \cite{zhao}.
	
	\begin{theorem}\label{var2}
		If $E[X] > 0$ and $Var(X) = o(E[X])^2$, then $X>0$ and $X \sim E[X]$ with probability $1-o(1)$.
	\end{theorem}
	
	\begin{theorem}\label{EX}
		For large enough $g(r)$, the existence of an $r$-regular link irregular graph is very likely.
	\end{theorem}
	\begin{proof}
		We use the same probabilistic model and random variable $X$ as in the proof of Theorem~\ref{first-moment}, where we heuristically assume that each of the $n$ vertices is assigned a link graph chosen uniformly at random from the $g(r)$ possible $r$-regular link types. Under the assumption that $n \ll g(r)$, we approximate $\Pr(I_v = 1) \approx \frac{1}{g(r)}$.
		
		Then for any vertex $v$, the probability that its assigned link is unique is:
		\[
		E[I_v] = \left(1 - \frac{1}{g(r)}\right)^{n-1} \approx e^{-\frac{n-1}{g(r)}} \quad \text{for large } g(r).
		\]
		Thus,
		\[
		E[X] = n e^{-\frac{n-1}{g(r)}}.
		\]
		
		Next, we estimate the variance:
		\[
		E[X^2] = \sum_{v} E[I_v^2] + \sum_{u \neq v} E[I_u I_v].
		\]
		Since $I_v^2 = I_v$, the first sum equals $E[X]$. The second term,  Under the assumption that $n \ll g(r)$, we approximate
		\[
		E[I_u I_v] \approx \left(1 - \frac{2}{g(r)}\right)^{n-2} \approx e^{-\frac{2(n-2)}{g(r)}}.
		\]
		So,
		\[
		\operatorname{Var}(X) = E[X^2] - (E[X])^2 \approx n e^{-\frac{n-1}{g(r)}} + n(n-1) e^{-\frac{2(n-2)}{g(r)}} - \left(n e^{-\frac{n-1}{g(r)}}\right)^2.
		\]
		
		Now consider:
		\[
		\frac{\operatorname{Var}(X)}{(E[X])^2} \approx \frac{1}{n} \left(e^{\frac{n-1}{g(r)}} - 1\right) \to 0 \quad \text{as } n \to \infty, \text{ assuming } n \ll g(r).
		\]
		
		By Chebyshev’s inequality (second moment method), this implies:
		\[
		\Pr(X > 0) \to 1 \quad \text{as } n \to \infty,
		\]
		which means that with high probability, there exists at least one assignment of links in which all vertices receive distinct links. This completes the proof.

	\end{proof}
	
	In the following, we provide an explicit counterexample for Conjecture \ref{conjecture}. 
	\begin{theorem} There exists a 7-regular link-irregular graph on 12 vertices. 
	\end{theorem}
	\begin{proof}
		Consider the following graph $G$ which is a 7-regular graph on 12 vertices. 
		
		\begin{figure}[htbp]
			\centering
			\includegraphics[width=0.45\textwidth]{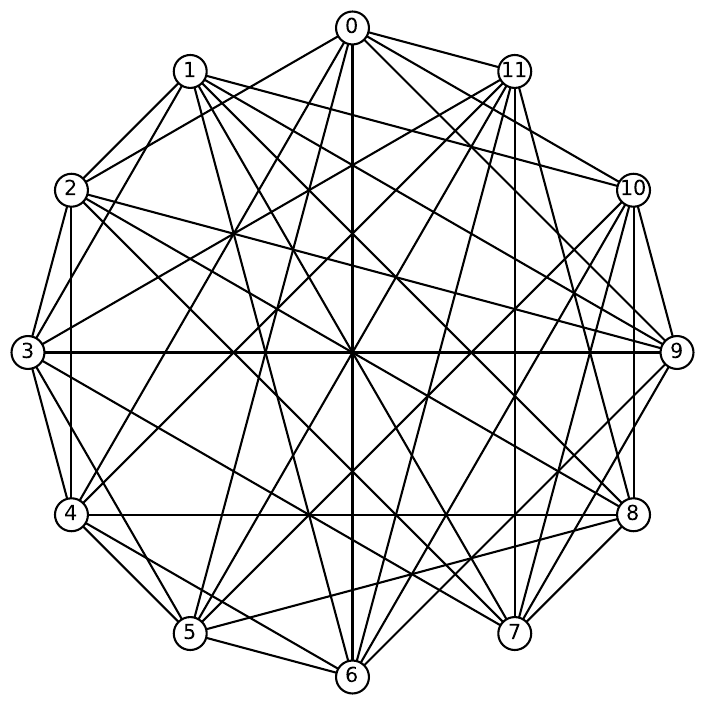}
			\caption{The graph \( G \), a 7-regular link-irregular graph on 12 vertices.}
			\label{fig:link-irregular}
		\end{figure}

		The links of the vertices are shown in Figure \ref{fig:links}.
		
		\begin{figure}[htbp]
			\centering
			\includegraphics[width=0.7\textwidth]{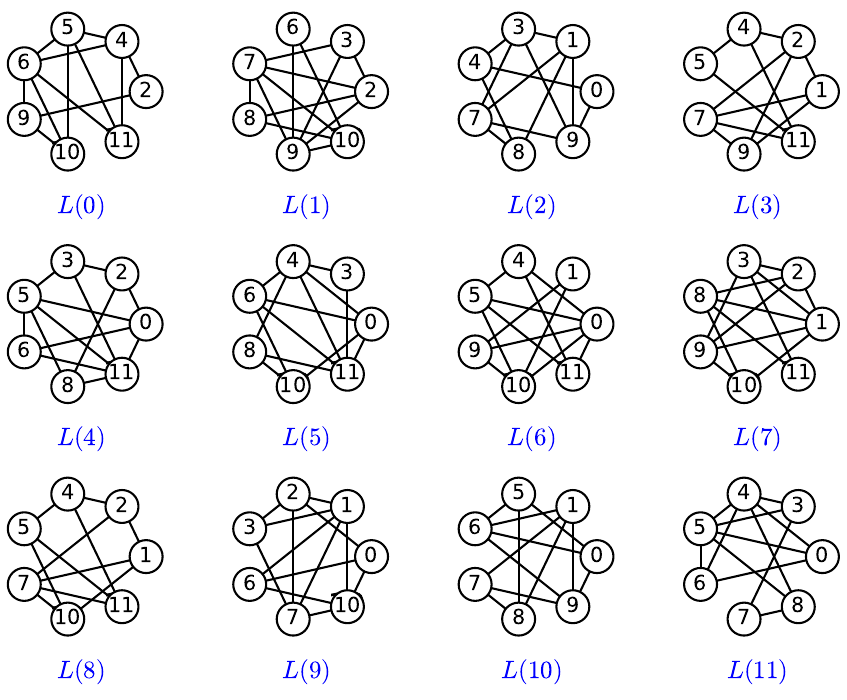}
			\caption{All links of the vertices of the graph $G$.}
			\label{fig:links}
		\end{figure}

		To verify that these links are non-isomorphic, we check their degree sets:
		\begin{center}
			\begin{tabular}{c    c}
				$L(0):  \{2,3,3,3,4,4,5\}$  \ \   &  \ \  $L(6):  \{2,3,3,3,4,4,5\}$ \\
				$L(1):  \{2,3,3,4,4,4,5\}$ \ \    &  \ \  $L(7):  \{2,3,4,4,4,4,5\}$  \\
				$L(2):  \{2,3,3,4,4,4,4\}$  \ \  & \ \  $L(8):  \{3,3,3,3,3,3,4\}$ \\
				$L(3):  \{2,3,3,3,3,4,4\}$ \ \  &  \ \   $L(9):  \{3,3,3,4,4,4,5\}$  \ \\
				$L(4):  \{3,3,3,3,4,5,5\}$  \ \  &  \ \   $L(10): \{3,3,3,3,4,4,4\}$  \\
				$L(5):  \{2,3,3,4,4,5,5\}$   \ \ &  \ \   $L(11): \{2,3,3,3,3,5,5\}$
			\end{tabular}

		\end{center}
		Among all vertex links, only the degree sets of $L(0)$ and $L(6)$ are identical; however, a closer inspection of their adjacency structures shows that they are non-isomorphic. Therefore, $G$ is a 7-regular link-irregular graph on 12 vertices.
	\end{proof}
	
	Note that the edge set of the counterexample is \\
	$\{ (3, 4), (3, 7), (3, 1), (3, 9), (3, 5), (3, 11), (3, 2), 
	(4, 6), (4, 5), (4, 8), (4, 11), (4, 2), (4, 0), 
	(7, 10), \\ (7, 2), (7, 9), (7, 1), (7, 11), (7, 8), 
	(6, 1), (6, 11), (6, 5), (6, 10), (6, 0), (6, 9), 
	(0, 2), (0, 5),  (0, 11),\\ (0, 10), (0, 9), 
	(2, 8), (2, 1), (2, 9), (5, 10), (5, 11), (5, 8),
	(10, 9), (10, 8), (10, 1), (1, 9), (1, 8), (11, 8) \}.$
	\ \\

	The graph in Figure~\ref{fig:link-irregular} was discovered using a SageMath/Python script that searches small regular graphs and checks for link-irregularity by testing non-isomorphism of all vertex links. The pseudocode is given below, and the full implementation can be found in our GitHub repository\footnote{\url{https://github.com/omidkhormali/link-irregular-graphs}}. Readers can use the provided SageMath code to find and identify other regular link-irregular graphs. 
	
	\begin{algorithm}[H]
		\caption{Search for $r$-regular link-irregular graphs}
		\begin{algorithmic}[1]
			\FOR{each $r$-regular graph $G$ of order $n$}
			\STATE Compute all links $L(v)$ for $v \in V(G)$
			\IF{all $L(u) \not\cong L(v)$ for $u \ne v$}
			\STATE \textbf{return} $G$ as a link-irregular graph
			\ENDIF
			\ENDFOR
		\end{algorithmic}
	\end{algorithm}
	
	We now examine bounds on the number of vertices in regular link-irregular graphs. We begin with a lemma.
	
	\begin{lemma}\label{high-degree}
		Any link-irregular graph with $n$ vertices has at most $\frac{n}{2}$ vertices of degree $n-2$.
	\end{lemma}
	\begin{proof}
		Let $G$ be a link-irregular graph with $n$ vertices, and suppose $u$ and $v$ are vertices of degree $n-2$ in $G$. If $u$ and $v$ are not adjacent, then they must have the same neighborhoods. But this contradicts the link-irregularity of $G$. Hence, $u \leftrightarrow v$ and we must have vertices $u'$ and $v'$ such that $u'$ and $u$ are nonadjacent, and likewise $v'$ and $v$ are nonadjacent. If $u' = v'$, then $u$ and $v$ have the same neighborhoods which this contradicts the link-irregularity of $G$. Then, we assume that $u' \neq v'$.\\
		Suppose we have vertices $v_1, \cdots, v_k$ in $G$ of degree $n-2$. As above, we also have vertices $v_1', \cdots, v_k'$ such that $v_i \nleftrightarrow v_i'$ and $v_i \leftrightarrow v_j$ if $i \neq j$. Since $v_i$ and $v_j$ must be adjacent for any $i$ and $j$, we have $v_i' \notin \{v_1, \cdots v_k\}$ for all $i$. Because $v_i' \neq v_j'$ for $i \neq j$, this means that $|\{v_1, \cdots, v_k, v_1', \cdots, v_k'\}| = 2k \leq n$. Hence $k \leq n/2$.
	\end{proof}
	
	Using this Lemma, we can find the following result on the number of vertices of regular link-irregular graphs.
	
	\begin{theorem}\label{9-no-regular}
		For $n \leq 9$, there are no regular link-irregular graphs on $n$-vertices.
	\end{theorem}
	
	\begin{proof}
		Let $G$ be a link-irregular graph on $n$ vertices. We consider the following cases. \\
		For $n \leq 6$, the result holds since there are no link-irregular graphs with fewer than 6 vertices by Theorem \ref{akbar1}, and there is only one unique link-irregular graph on 6 vertices (Figure \ref{fig:unique-6}), which is not regular. \\
		For $n = 7$, by Theorem \ref{no-regular}, $G$ cannot be 1, 2, 3, or 4-regular. By Lemma \ref{high-degree}, $G$ cannot be 5-regular either. Also, it is clear that $G$ cannot be 6-regular since that would imply $G$ is a $K_7$, which contradicts the link-irregularity of $G$.  \\
		For $n = 8$, as before, $G$ cannot be 1, 2, 3, 4, 6, or 7-regular. If $G$ is regular, it must be 5-regular. But in this case, $\bar G$ is 2-regular and has a non-trivial automorphism by rotating a component; but this is impossible since it would imply a non-trivial automorphism of $G$, so $G$ is not regular. \\
		For $n = 9$, $G$ cannot be 1, 2, 3, 4, 7, or 8-regular by the same arguments used before. Also, $G$ cannot be 5-regular either since $n$ is odd. The only possibility is that $G$ is 6-regular. But then in this case, $\bar G$ is 2-regular and has a non-trivial automorphism, which is impossible.
	\end{proof}
	
	Based on our observations in smaller graphs, we propose the following conjecture.
	\begin{conj}
		There exists a regular link-irregular graph on n vertices if and only if $n \geq 12$.
	\end{conj}
	
	\section{Edge Bounds in Link-Irregular Graphs}
	
	We next find bounds for the number of edges in a link-irregular graph.
	\begin{theorem}
		If $G$ is a link-irregular graph with $n$ vertices, then $e(G) \leq \dfrac{2n^2 - 5n + 4}{4}$.
	\end{theorem}
	
	\begin{proof}
		Any two vertices of degree $n-1$ have the same neighborhood. By this and Lemma \ref{high-degree}, we have at most one vertex of degree $n-1$ and at most $n/2$ vertices of degree $n-2$. The rest have degree at most $n-3$. Then
		\[
		e(G) = \frac{1}{2}\sum_{v \in V(G)} d(v) \leq \frac{1}{2}[ (n-1) + \frac{n}{2}(n-2) +(n-(\frac{n}{2}-1))(n-3)] = \frac{2n^2 -5n +4}{4}.
		\]
	\end{proof}
	
	We now establish a lower bound. We note that there are 1, 2, 4 isomorphism classes of graphs on 1, 2, and 3 vertices, respectively.
	\begin{theorem}
		If $G$ is a link-irregular graph with $n$ vertices, then $e(G) \geq 2n - 5$.
	\end{theorem}
	
	\begin{proof}
		If $G$ is the unique link-irregular graph with six vertices (shown in Figure \ref{fig:unique-6}), the result follows by counting its edges.\\
		Suppose $n \geq 7$. Then, at most, 1, 2, and 4 vertices have degrees 1, 2, and 3, respectively. The rest have a degree at least 4. From this, we have
		\[
		e(G) = \frac{1}{2} \sum_{v \in V(G)}d(v) \geq \frac{1}{2}[1 + 2(2) + 3(4) + 4(n-7)] = \frac{1}{2} (4n-11).
		\]
		Since $e(G)$ must be an integer. We have $e(G) \geq 2n - 5$.
	\end{proof}
	
	We can asymptotically strengthen this bound without significant additional work.
	
	\begin{theorem}
		Let $f(n)$ denote the minimum number of edges in a link-irregular graph on $n$ vertices. Then $f(n) = \Omega (n\sqrt {\log n})$.
	\end{theorem}
	
	\begin{proof}
		We note first that an upper bound on the number of isomorphism classes of graphs on $h$ vertices is $2^{h\choose2}$. Hence, any link-irregular graph $G$ has at most $2^{h\choose2}$ vertices of degree $h$.
		Given a number $n$ of vertices, we may pick $k$ such that $2^{k\choose2} \leq n < 2^{{k+1}\choose 2}$. For every value $d$ such that $1 \leq d < k$, a link-irregular graph on $n$ vertices has at most $2^{d\choose2}$ vertices of degree $d$. Such a graph also has at least $n - \sum_{d=1}^{k-1} 2^{d\choose2}$ vertices of degree at least $k$. We have
		\[
		\sum d(v) \geq \sum_{d = 1}^{k-1}d2^{d\choose2} + k(n - \sum_{d=1}^{k-1} 2^{d\choose2}) = kn - \sum_{d = 1}^{k-1}(k-d)2^{d\choose2}.
		\]
		Since $2^{k\choose2} \leq n < 2^{{k+1}\choose 2}$, $\frac{k^2-k}{2} \leq \log n < \frac{k^2+k}{2} $. Hence, as $n \rightarrow \infty$, $k \rightarrow \sqrt{2\log \space n}$. In addition, we have
		\[\sum_{d = 1}^{k-1}(k-d)2^{d\choose2} < k\sum_{d = 1}^{k-1}2^{d\choose2} < k\sum_{i = 1}^{{k-1}\choose2} 2^i = k(\frac{1-2^{{k-1}\choose2}}{1-2}) = k(2^{{k-1}\choose2} - 1).
		\]
		Since $n \geq 2^{k\choose2}$, the term $nk$ dominates $k(2^{{k-1}\choose2} - 1)$ and therefore dominates the term $\sum_{d = 1}^{k-1}(k-d)2^{d\choose2}$. Combining this with the first inequality and the degree-sum formula, we obtain $f(n) = \Omega(\frac{n\sqrt{2 \log n}}{2}) = \Omega(n\sqrt{\log n})$. 
	\end{proof}
	
	\section{Planarity of Link-Irregular Graphs}
	
	The unique link-irregular graph on 6 vertices in Figure \ref{fig:unique-6} is planar, which motivated us to explore the planarity property of link-irregular graphs. Surprisingly, we found that almost all other link-irregular graphs we identified are non-planar, suggesting that planarity is a rare property in this class of graphs. From \cite{west}, we have the following result.
	
	\begin{lemma}[\cite{west}]\label{planar-lemma}
		If G is a simple planar graph with at least three vertices, then $e(G) \leq 3n(G) - 6$.
	\end{lemma}
	
	We use Lemma \ref{planar-lemma} to prove the following theorem.
	\begin{theorem}\label{planar}
		If $G$ is a link-irregular graph with $n > 277$ vertices, then $G$ is not planar.
	\end{theorem}
	
	\begin{proof}
		The number of isomorphism classes of graphs on 1, 2, 3, 4, 5, and 6 vertices are 1, 2, 4, 11, 34, and 156, respectively. Hence, any link irregular graph with more than $1+2+4+11+34+156 = 208$ vertices satisfies the inequality
		\[
		e(G) = \frac{1}{2}\sum_{v\in G} d(v) \geq \frac{1}{2}[1+2(2)+3(4)+4(11)+5(34)+6(156)+7(n-108)] = \frac{7n-289}{2}. 
		\]
		For $n>277$, $\frac{7n-289}{2} > 3n - 6$. Hence, by Lemma \ref{palanar-lemma}, $G$ is not planar.
	\end{proof}
	
	We can conclude the following corollary using Theorem \ref{planar}.
	\begin{cor}
		All but finitely many link-irregular graphs are non-planar.
	\end{cor}
	
	The following result is about existence of planar regular link-irregular graphs. But we first mention to two results in \cite{west}; the first is a consequence of Kuratowski's Theorem, and the other is about the triangulation of planar graphs.
	\begin{cor}[Consequence of Kuratowski’s Theorem~\cite{west}]\label{k5-and-k33}
		If a graph \( G \) contains \( K_5 \) or \( K_{3,3} \) as a subgraph, then \( G \) is non-planar.
	\end{cor}
	
	\begin{prop}[\cite{west}]\label{traiangulation}
		For a simple $n$-vertex planar graph $G$, the following are equivalent.
		\begin{enumerate}[label=(\alph*), itemsep=-5pt]
			\item $G$ has $3n-6$ edges.
			\item $G$ is a triangulation. (A triangulation is a graph where every face boundary is a cycle of length 3.)
			\item $G$ is a maximal planar graph.
		\end{enumerate}
	\end{prop}
	
	Now, our result is in the following.
	\begin{theorem}
		There is no regular planar link-irregular graph.
	\end{theorem}
	\begin{proof}
		Suppose $G$ is a $r$-regular planar link-irregular graph. Using Lemma \ref{palanar-lemma}, we can conclude that $r\leq 5$ since 
		$
		\sum_{v\in V(G)}{d(v)}= nr = 2e(G)\leq 2(3n-6) \Rightarrow{} r \leq 6 - \frac{12}{n}.
		$
		Also, by Theorem \ref{no-regular}, the only possibility for $r$ is 5, which we now consider. Since $G$ is a 5-regular link irregular graph, its number of vertices cannot be more than 34, which is the number of non-isomorphic graphs on 5 vertices. Also, since $G$ is 5-regular, the number of vertices of $G$ must be an even number.
		In \cite{akbar_paper2}, all these 34 non-isomorphic graphs of order 5 are provided. You can see all of them in Figure \ref{fig:order5}.
		\begin{figure}[H]
			\centering
			\includegraphics[scale=0.095]{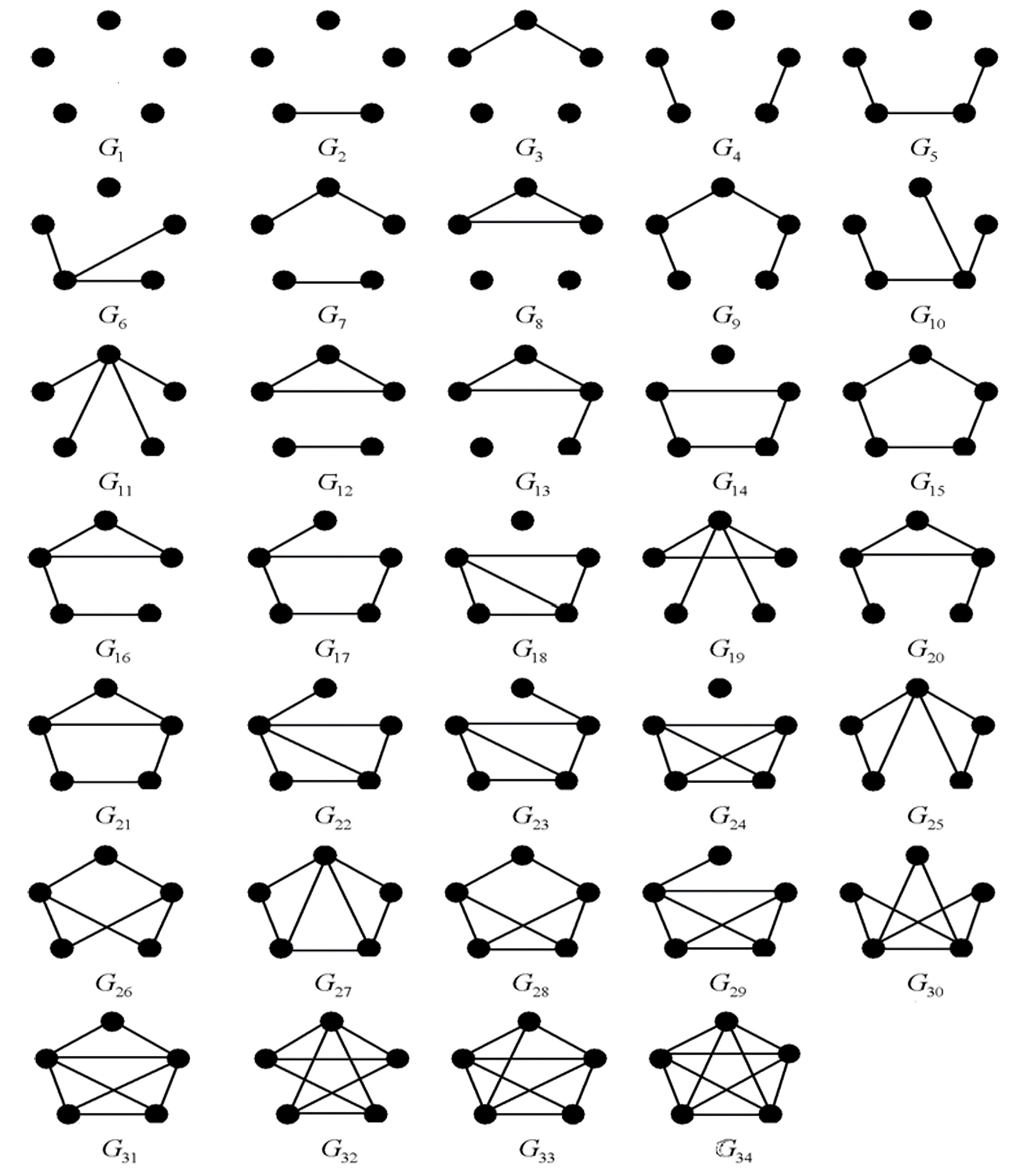}
			\caption{All non-isomorphic graphs of order 5 \cite{akbar_paper2}.}
			\label{fig:order5}
		\end{figure}
		
		Since $G$ is planar, $G$ cannot have $K_5$ or $K_{3,3}$ as a subgraph by Corollary \ref{k5-and-k33}. Consequently, the links of vertices in $G$ cannot have $K_5$, $K_4$, or $K_{3,2}$ as subgraphs. In addition, no vertex in the link of any vertex can have degree 4; otherwise, the vertex of degree 4 in the link and the link owner vertex would have isomorphic links, which contradicts the link-irregularity of the graph. Hence, $G$ does not have any of the following links from the figure
		\[G_{11}, G_{19}, G_{21}, G_{22}, G_{24}, G_{25}, G_{26}, G_{27}, G_{28}, G_{29}, G_{30}, G_{31}, G_{32}, G_{33}, G_{34}.\]
		
		Then, by this list and Theorem \ref{9-no-regular}, the possible number of vertices for $G$ are 10, 12, 14, 16 and 18. For $n=10$, the number of edges must be 25. But, by Lemma \ref{planar-lemma}, we have $3n(G)-6 = 3(10)-6 = 24 <25=e(G)$ which contradicts the planarity of $G$.\\
		For $n=12$, we have $3n(G)-6 = 3(12)-6 = 30 =e(G)$. Then, by Proposition \ref{traiangulation}, $G$ is a triangulation and each face has length 3. Also, as mentioned in~\cite{west}, there is a unique graph (up to graph isomorphism) with the same properties as $G$, namely the icosahedron graph, which is shown in the following figure, which is clearly not a link-irregular graph.
		\begin{figure}[H]
			\centering
			\includegraphics[scale=0.25]{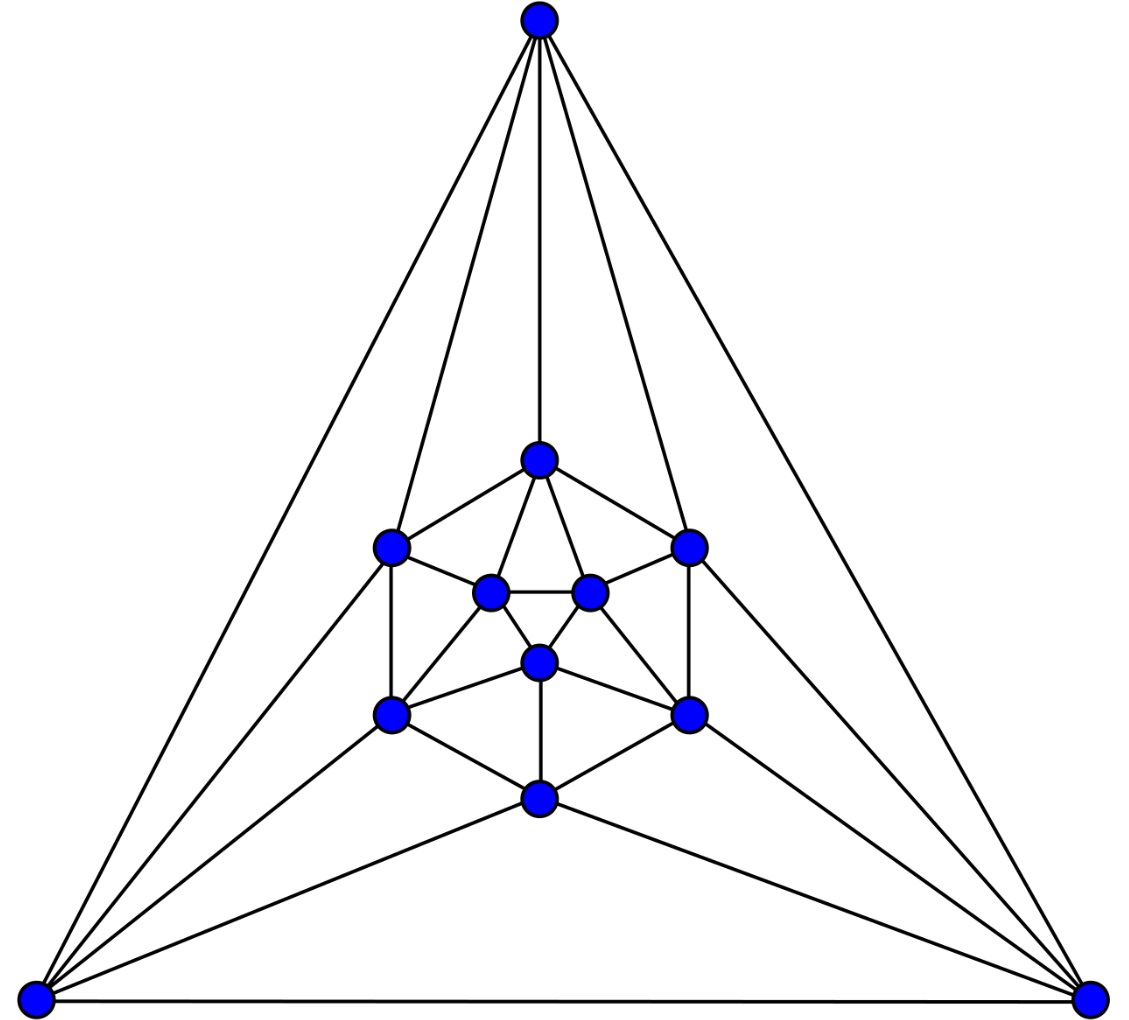}
			\caption{The icosahedron graph.}
			\label{fig:icosahedron}
		\end{figure}
		
		For $n=14$, in \cite{owens}, it is proved that the 5-regular graph on 14 vertices is not planar. Then, there clearly can not be a 5-regular planar link-irregular graph on 14 vertices.\\
		For $n=16$ and 18, in \cite{mahdieh}, it is stated that there exist unique 5-regular planar graphs (up to isomorphism) on 16 and 18 vertices, respectively. Also, These planar graphs can be generated using the online interface at \footnote{\url{https://combos.org/plantri}}, which provides access to the Plantri program developed by Gunnar Brinkmann and Brendan McKay. The website outputs adjacency lists, where the neighbors of each vertex are listed in clockwise (cw) order according to a planar embedding. We used a custom SageMath script to visualize the graphs; this script is available on the GitHub repo. The graphs are shown in the figure \ref{fig:moreplanargraphs}.
		
		\begin{figure}[H]
			\centering
			\begin{subfigure}[t]{0.35\textwidth}
				\centering
				\includegraphics[width=\textwidth]{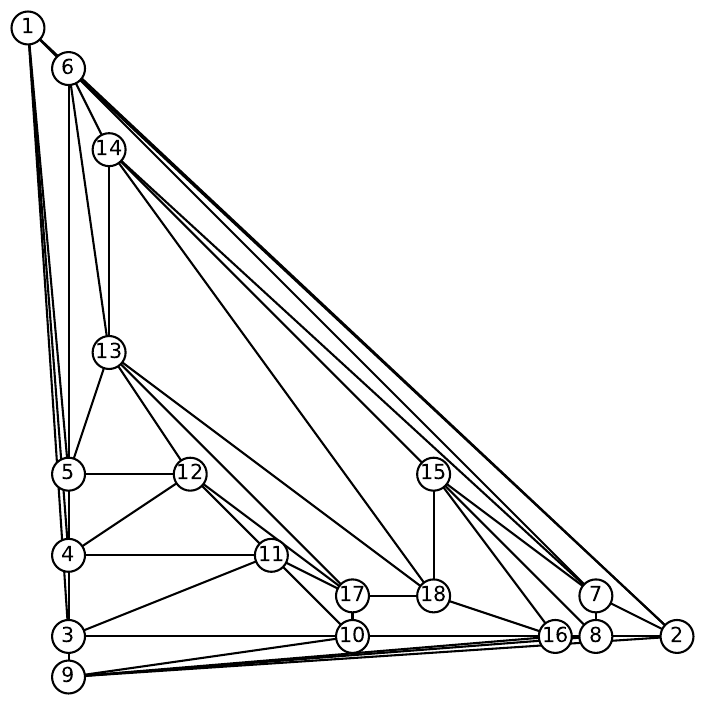}
				\caption{5-regular planar graph with 18 vertices}
				\label{fig:planar18}
			\end{subfigure}
			\hfill
			\begin{subfigure}[t]{0.35\textwidth}
				\centering
				\includegraphics[width=\textwidth]{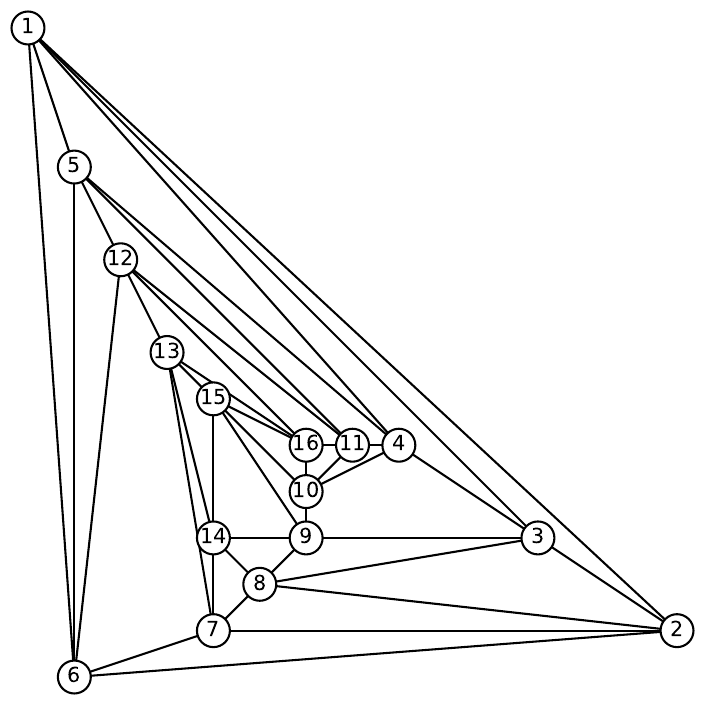}
				\caption{5-regular planar graph with 16 vertices}
				\label{fig:planar16}
			\end{subfigure}
			\caption{The unique 5-regular planar graphs on 16 and 18 vertices.}
			\label{fig:moreplanargraphs}
		\end{figure}
		
		By inspecting the vertex connectivity in Figure \ref{fig:moreplanargraphs2} that are graphs in Figure \ref{fig:moreplanargraphs} by rearranging the location of vertices, one sees that the links of the vertices 1 and 2 in both graphs are isomorphic, so neither of these are link-irregular. \\
		
		\begin{figure}[htbp]
			\centering
			\begin{subfigure}[t]{0.35\textwidth}
				\centering
				\includegraphics[width=\textwidth]{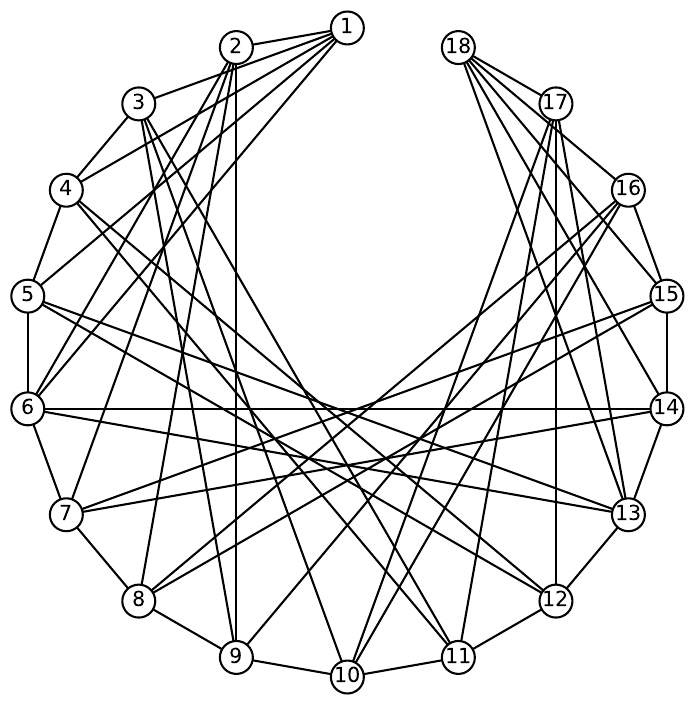}
				\caption{5-regular planar graph with 18 vertices}
				\label{fig:planar18}
			\end{subfigure}
			\hfill
			\begin{subfigure}[t]{0.35\textwidth}
				\centering
				\includegraphics[width=\textwidth]{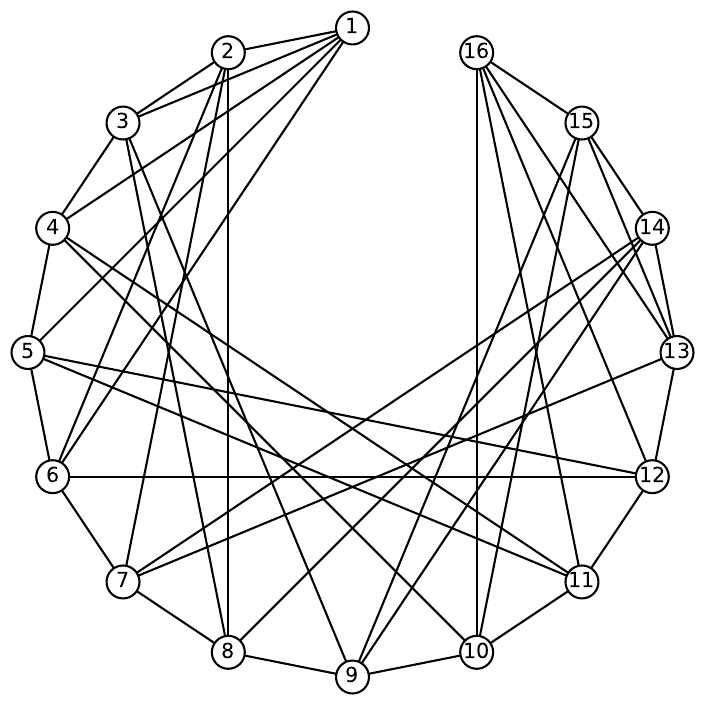}
				\caption{5-regular planar graph with 16 vertices}
				\label{fig:planar16}
			\end{subfigure}
			\caption{The unique 5-regular planar graphs on 16 and 18 vertices.}
			\label{fig:moreplanargraphs2}
		\end{figure}
		Hence, no 5-regular link-irregular graph is planar.
		
	\end{proof}
	
	The edge set of the unique 5-regular planar graph on 16 vertices follows:\\
	$\{(1,2), (1,3), (1,4), (1,5), (1,6),
	(2,6), (2,7), (2,8), (2,3),
	(3,8), (3,9), (3,4),
	(4,10), (4,11),\\ (4,5),
	(5,11), (5,12), (5,6),
	(6,12), (6,7),
	(7,13), (7,14), (7,8),
	(8,14), (8,9),
	(9,14), (9,15),\\ (9,10),
	(10,15), (10,16), (10,11),
	(11,16), (11,12),
	(12,16), (12,13),
	(13,16), (13,15), (13,14),\\
	(14,15),
	(15,16)\}$.
	\ \\
	
	The edge set of the unique 5-regular planar graph on 18 vertices follows:\\
	$\{(1,2), (1,3), (1,4), (1,5), (1,6),
	(2,6), (2,7), (2,8), (2,9),
	(3,9), (3,10), (3,11), (3,4),
	(4,11),\\ (4,12), (4,5),
	(5,12), (5,13), (5,6),
	(6,13), (6,14), (6,7),
	(7,14), (7,15), (7,8),
	(8,15), (8,16),\\ (8,9),
	(9,16), (9,10),
	(10,16), (10,17), (10,11),
	(11,17), (11,12),
	(12,17), (12,13),
	(13,17),\\ (13,18), (13,14),
	(14,18), (14,15),
	(15,18), (15,16),
	(16,18),
	(17,10), (17,18), (17,12),\\ (17,11),
	(18,13), (18,17), (18,16), (18,15), (18,14)\}$.
	
	\section*{Acknowledgments}
	We thank Clark Kimberling for his useful comments and suggestions. This research was supported by the UExplore Undergraduate Research Program at the University of Evansville.

\end{document}